\documentclass[a4paper,12pt]{article}
\usepackage{lscape}
\usepackage{a4,amssymb,amsmath,amsthm,url,lineno,threeparttable}
\usepackage{bezier,amsfonts,amssymb,graphicx,amsthm,url}
\usepackage[english]{babel}
\title{On small balanceable, strongly-balanceable and omnitonal graphs}
\date{}
\begin{document}
\newtheorem{theoremnum}{Theorem}[section]
\newtheorem*{lemma*}{Lemma}
\newtheorem{theorem}{Theorem}
\renewcommand*{\thetheorem}{\Alph{theorem}}

\newtheorem{definition}{Definition}[section]
\newtheorem{observation}[theoremnum]{Observation}
\newtheorem{proposition}[theorem]{Proposition}
\newtheorem{corollary}[theoremnum]{Corollary}
\newtheorem{lemma}[theorem]{Lemma}
\newcommand{\bal}{{\rm bal}}
\newcommand{\sbal}{{\rm sbal}}
\newcommand{\ot}{{\rm ot}}
\newcommand{\ex}{{\rm ex}}
\newtheorem{Ex}[theoremnum]{$\rhd$ Example}

\DeclareGraphicsExtensions{.pdf,.png,.jpg}

\author{Yair Caro \\ Department of Mathematics\\ University of Haifa-Oranim \\ Israel \and Josef  Lauri\\ Department of Mathematics \\ University of Malta
\\ Malta \and Christina Zarb \\Department of Mathematics \\University of Malta \\Malta }

\maketitle

\begin{abstract}

In Ramsey theory for graphs we are given a graph $G$ and we are required to find the least $n_0$ such that, for any $n\geq n_0$, any red/blue colouring of the edges of $K_n$ gives a subgraph $G$ all of whose edges are blue or all are red. Here we shall be requiring that, for  any red/blue colouring of the edges of $K_n$, there must be a copy of $G$  such that its edges are partitioned equally as red or blue (or the sizes of the colour classes differs by one in the case when $G$ has an odd number of edges). This introduces the notion of balanceable graphs and the balance number of $G$ which, if it exists, is the minimum integer bal$(n, G)$ such that, for any red/blue colouring of $E(K_n)$ with more than bal$(n, G)$ edges of either colour, $K_n$ will contain a balanced coloured copy of $G$ as described above. 
	
	This parameter was introduced by Caro, Hansberg and Montejano in   \cite{2018arXivCHM}. There, the authors also introduce the strong balance number sbal$(n,G)$ and the more general omnitonal number ot$(n, G)$ which requires copies of $G$ containing a complete distribution of the number of red and blue  edges over $E(G)$.
	
	In this paper we shall catalogue bal$(n, G)$, sbal$(n, G)$ and ot$(n,G)$ for all graphs $G$ on at most four edges. We shall be using some of the key results of Caro et al, which we here reproduce in full, as well as some new results which we prove here.  For example, we shall prove that the union of two bipartite graphs with the same number of edges is always balanceable.
\end{abstract}
\section{Introduction}

The problem we consider here, introduced in \cite{2018arXivCHM}, lies in the intersection of several graph theory problems such as Ramsey Theory, Extremal Graph Theory (Turan numbers) and Zero-sum Ramsey Theory.  It can be described as follows:  we first suppose that there is a 2-edge-colouring $f:E(K_n) \to \{red,blue\}$, and we denote by $R=R_f$ and $B=B_f$ the set of edges of $K_n$ coloured red or blue respectively. For short we shall also denote by $R$ and $B$ the graphs induced by the edge sets $R$ and $B$, respectively.  A subgraph $G$ of such a coloured complete graph is said to be $(r,b)$-coloured if $r$ edges of $G$ are coloured red and $b$ edges are coloured blue with $r+b=e(G)$, where $e(G)$ denotes the number of edges of $G$. We denote by  $deg_{red}(v)$   the number of  vertices adjacent to $v$ by a red edge, $N_{red}(v) = \{ u : uv \mbox{ is a red edge} \}$ so that $deg_{red}(v)=|N_{red}(v)|$, while $N_{red}[v]= N_{red}(v) \cup \{v \}$.   For other standard graph theoretic notation we refer to \cite{west2017introduction}.


In  Ramsey Theory for graphs, we require that one of $R$ or $B$ contains a particular graph, say a complete graph, and we ask what is the smallest value of $n$ such that any 2-edge-colouring $f$ gives us the graph we want either in $R$ or in $B$. In this  paper, inspired by   \cite{2018arXivCHM}, the graphs we are searching for will be required to have some particular mix of colours on its edge-set.  There are three main problems we consider:

\subsubsection*{Balanceable Graphs}

For a given graph $G$ we say that the colouring contains a \emph{ balanced} copy of $G$ if $f$ induces a coloured copy of $G$ in which the number of edges in each colour is equal (if $G$ has an even number of edges) or differs by one.

We therefore let, for $n \geq |V(G)|$,  $\bal(n,G)$ be the minimum integer, if it exists, such that any 2-edge-colouring of $K_n$ with $\min\{|R|,|B|\} > \bal(n,G)$ contains a balanced copy of $G$.   If $\bal(n,G)$ exists for every sufficiently large $n$, we say that $G$ is balanceable.

If $G$ has an odd number of edges we then introduce the notion of \emph{strong balance}.  That is, for $n \geq |V(G)|$, we let $\sbal(n,G)$ be the minimum integer, if it exists, such that any 2-edge-colouring of $K_n$  with $\min\{|R|,|B|\} > \sbal(n,G)$ contains  both a $(\lfloor \frac{e(G)}{2} \rfloor,\lceil \frac{e(G)}{2} \rceil)$-coloured  and $(\lceil \frac{e(G)}{2} \rceil ,\lfloor \frac{e(G)}{2} \rfloor)$-coloured copy of $G$. If $\sbal(n,G)$ exists for every sufficiently large $n$, we say that $G$ is strongly-balanceable.


\subsubsection*{Omnitonal Graphs}

\emph{Omnitonal} graphs are those graphs $G$ for which different copies of $G$ appear in a 2-edge-coloured complete graph such that all the copies carry between them all possible distributions of the two colours on the edges of $G$.  More formally, we define, for a given graph $G$,  and  for $n \geq |V(G)|$, $\ot(n,G)$ to be the minimum integer, if it exists, such that  any 2-edge-colouring of $K_n$ with $\min\{|R|,|B|\} > \ot(n,G)$ contains an $(r,b)$-coloured copy of $G$ for any $r \geq 0$ and $b \geq 0$ such that $r+b=e(G)$.  If $\ot(n,G)$ exists for every sufficiently large $n$, we say that $G$ is omnitonal.

A source of motivation  for \cite{2018arXivCHM}, that belongs to the recent developments in zero-sum extremal problems,  is the close connection between the concepts of balanceable and omnitonal graphs and zero-sum problems with weights over $\{ -p,q \}$  and in particular over $\{ -1 ,1\}$  (see remark 1.3 in \cite{2018arXivCHM}) \cite{augspurger2016avoiding,berger2019analogue,bowen2019colored,caro2019amoeba,2018arXivCHM,caro2019zero2,caro2019zero1,caro2016ero,girao2019tur,robertson2018zerosum2,robertson2018zerosum1,sun2019zero}.

\medskip 

\noindent There are two main goals  as background to the present paper:  the first one is to compute the functions $\sbal(n,G)$, $\bal(n,G)$  and $\ot(n,G)$  for all the graphs with up to 4 edges, checking which  of the  results already obtained in \cite{2018arXivCHM} can be applied in this task.  When  no such result from \cite{2018arXivCHM} does the job we give our complementary ad-hoc theorems that  allow us to complete the various tables. 

The second goal of this paper is to find out, while working on completing the tables, whether we can gain some further insight not covered in \cite{2018arXivCHM}.  In fact we have found at least one such instance.  It is mentioned in \cite{2018arXivCHM} that  not all bipartite graphs are balanceable (or strongly balanceable).  We    nevertheless here prove that  if $G$ and $H$ are bipartite graphs with $e(G) = e(H)$, then  $G \cup  H$ is balanceable.  We remark that  several theorems about the union of balanceable or omnitonal graphs are known (and will  probably appear in [3]), but all of them assume that at least one of the graphs is  balanceable or omnitonal. 

\medskip

\noindent Our paper is organized as follows:

In section 2  we collect the theorems proved in \cite{2018arXivCHM} which we need here and also give our complementary results that  allow us to complete the first task   - namely to compute $\sbal(n,G)$, $\bal(n,G)$  and $\ot(n,G) $  for all the graphs on up to  four edges.

 In section 3 we  prove the union theorem  mentioned above, and give an example to illustrate its use.  We then introduce the triple property, and use this,  Theorem \ref{thm:ramsey} and the fact that the graphs $tK_2$ are amoebas (to be defined later), to show  that  for  $n \geq 7t  - 1$, $\sbal( n, (2t-1)K_2)  = \bal(n,  2tK_2)  =  \bal(n, (2t+1)K_2)  =  (t -1)(n-t+1)  + \binom{t-1}{2}.$

Section 4 contains  the tables  with the results of our computations.

\section{Theorems and Further Definitions}

In this section we  state those theorems from \cite{2018arXivCHM} which we shall use in the presentation of our results, and some of the definitions required to state these theorems.  We shall also present a few results  which do not appear in \cite{2018arXivCHM}.

For notation not defined here we refer the reader to \cite{west2017introduction}.  We here note that we shall oftern denote the edge $\{a,b\}$ by $ab$.

\subsection{Known Results}

The results presented here  are  taken from \cite{2018arXivCHM} and  \cite{caro2019amoeba}.

\begin{theorem}[Theorem 2.8 in \cite{2018arXivCHM}] \label{omnibip}
Omnitonal graphs are bipartite.
\end{theorem}

\begin{theorem}[Theorem 2.10 in \cite{2018arXivCHM}]
Every tree  is omnitonal.
\end{theorem}

An edge replacement is defined as follows:  given a graph $G$ embedded in a complete graph $K_n$ where $ n \geq |V(G)|$, we say that $H \simeq G$ (also embedded in $K_n$) is obtained from $G$ by an {\em edge-replacement} if there is some $e_1 \in E(G)$ and $e_2 \in E(K_n)\backslash E(G)$, such that $H=(G-e_1)+e_2$.   A graph $G$ is called an {\em amoeba} if there exists $n_0=n_0(G) \geq |V(G)|$, such that for all $n \geq n_0$ and any two copies $F$ and $H$ of $G$ in $K_n$, there is a chain $F=G_0,G_1,\ldots,G_t=H$ such that for every $i \in \{1,2,\ldots,t\}$, $G_i\simeq G$ and $G_i$ is obtained from $G_{i-1}$ by an edge-replacement.  For example,  it is easy to see that $tK_2$ is an amoeba  for $t \geq 1$,  and a little more effort reveals that  $P_k$  the path on $k$ vertices is also an amoeba.  A paper developing in depth structures and properties of amoebas is under preparation \cite{caro2019amoeba}.


For a given graph $G$, we denote by $R(G, G)$ the 2-colour Ramsey number, that is, the minimum integer $R(G, G)$ such that, whenever $n \geq R(G, G)$, any 2-edge-colouring of $E(K_n) = E(R) \cup
E(B)$ contains either a red or a blue copy of $G$. For a given graph $G$, we denote by $\ex(n, G)$
the Turan number for $G$, that is, the maximum number of edges in a graph with $n$ vertices
containing no copy of $G$ \cite{bollobas2004extremal,furedi2013history}.

\begin{theorem}[Theorem 2.14 in \cite{2018arXivCHM}] \label{bipam}
There is some $n_0=n_0(G)$ such that every bipartite amoeba $G$  is omnitonal with $\ot(n,G)=\ex(n,G)$ provided $n$ is large enough to fulfil $\binom{n}{2} \geq 2\ex(n,G)+1$ and $n\geq n_0$.
\end{theorem}

\begin{theorem}[Theorem 2.15 in \cite{2018arXivCHM}] \label{ambal}
Every amoeba $G$  is balanceable/strongly balanceable.
\end{theorem}

The final two results we state consider stars.

\begin{theorem}[Theorem 3.2 in \cite{2018arXivCHM}]
Let $k \geq 2$ and $n$ be integers with $k$ even and such that $n \geq \max\{3,\frac{k^2}{4}+1\}$.  Then \[\bal(n,K_{1,k}) = n\left (\frac{k}{2}-1 \right ) - \frac{k^2}{8} +\frac{k}{4}.\]
\end{theorem}

\begin{theorem}[Theorem 4.1 in \cite{2018arXivCHM}] \label{omnistar}
Let $n$ and $k$ be positive integers such that $n \geq 4k$.  Then \[
\ot(n,K_{1,k}) =
\begin{cases}
\left \lfloor \frac{(k-1)}{2}n \right \rfloor , & \text{for } k\leq 3,\\
(k-2)n - \frac{k^2}{2}+\frac{3}{2}k-1 & \text{for } k \geq 4.
\end{cases}
\]
\end{theorem}

The following is a simple lemma about amoebas that is observed in \cite{caro2019amoeba}.  We give the proof for completeness:

 \begin{lemma}[Lemma in \cite{caro2019amoeba}] \label{lemmadeg}

Let $G$ be an amoeba without isolated vertices.  Then $\delta(G) = 1$ and, for every  k, $1 \leq k \leq \Delta(G)$, there is a vertex $v$  in $G$ with $deg(v) = k$. 
\end{lemma}

\begin{proof}
Let $G$ be an amoeba.  Consider $K_n$ where $n$ sufficiently large and let $H$ be a copy of $G$  such that $H$ and $G$ are vertex disjoint in $K_n$. Let $v$ be a vertex in $G$ such that $deg(v) = \Delta(G)$.  Suppose $u$ is the vertex in $H$  to which $v$ is to arrive via edge-replacement.
 
So initially $deg(u) = 0$ and clearly the first time an edge is replaced to be incident with $u$, $deg(u) = 1$ ---  but that means that in $G$ there is a vertex of degree equal to 1. Now as the process of edge-replacements that carry $v$ to $u$  continues,  then every time the degree of $u$  can only increase by 1. So all the numbers between 1 and $\Delta(G)$ are present as degrees of $u$  along the process, and for each $k$,  $1 \leq  k \leq \Delta$, there must therefore be a vertex in $G$ that has degree $k$.
\end{proof}

Finally, we require the following definitions as described in \cite{2018arXivCHM}.  Let $t$ and $n$ be integers with $1 \leq t < n$. A 2-edge-coloured complete graph $K_n$ is said to be of \emph{type-$A(t)$} if the edges of one of the colours induce a complete graph $K_t$, and it is of \emph{type-$B(t)$} if the edges of one of the colours induce a complete bipartite graph $K_{t,n-t}$.  A \emph{type-$A(t)$} \emph{colouring} (respectively  \emph{type-$B(t)$} \emph{colouring}) of $K_n$ is called \emph{balanced} if the number of red edges equals the number of blue edges.
 The following lemma is used here when we consider whether the graph $C_4$ is omnitonal and $K_3$  strongly balanceable.

\begin{lemma}[Lemma 3.1 and Lemma 3.2 in \cite{caro2019zero2}, Lemma 2.3 in \cite{2018arXivCHM}] \label{typecolouring}
\leavevmode
\begin{enumerate}
\item{For infinitely many positive integers $n$, we can choose $t = t(n)$ in a way that the type-A(t) colouring of $K_n$ is balanced.}
\item{ For infinitely many positive integers $n$, we can choose $t = t(n)$ in a way that the type-B(t) colouring of $K_n$ is balanced.}
\end{enumerate}
\end{lemma}

Hence, such a type-$B(t)$  balanced colouring prevents an $(r,b)$-colouring of $C_4=K_{2,2}$ where $r +b = 4$ and where both $r$ and $b$ are odd, showing that $C_4$ is not omnitonal.

Also, such a type-$A(t)$ balanced colouring with edges forming the induced $K_t$ prevents a $(2,1)$-coloured $K_3$, showing that $K_3$ is not strongly balanced.

\begin{observation}
It is a  well-known simple fact  that the only graphs on at least three edges  containing no two independent edges are $K_3 $ and $K_{1,k}$ for $k \geq 3$.  Therefore if a graph on $n \geq 4$ vertices has at least $n$ edges then it must have at least one pair of independent edges.
\end{observation}

\subsection{New Results}
We now give some direct proofs of results which were not covered by the above theorems.

\begin{theoremnum}  \label{omnithm}
For $n \geq 10$, \ot$(n,K_{1,3} \cup  K_2)  = n$.
\end{theoremnum}

\begin{proof}

\noindent {\em Lower bound:} For $n \geq 10$  we have to show a  colouring  with $\min \{ |R|, |B| \}  = n$ and some $r ,b \geq 0$  such that $r+b = 4$  but no  copy of $K_{1,3} \cup  K_2$ has $r$ red edges and $b$ blue edges.

So, for $n \geq 10$, any colouring in which  the red colour forms a 2-factor in $K_n$ avoids a red $K_{1,3}$, hence avoids a red $K_{1,3} \cup  K_2$.
Observe that, for $n = 9$, a colouring in which the red edges form  a clique $K_5$  and the rest are blue avoids a red $K_{1,3} \cup  K_2$  but has 10 red edges hence the restriction to  $n \geq 10$  is necessary.   

\medskip\noindent
{\em Upper bound:}  We have to show that,  for $n \geq 10$,  any colouring of $E(K_n)$  in which $\min \{ |R|, |B| \}  \geq n+1$  contains an $(r,b)$-coloured  copy of $K_{1,3} \cup  K_2$  for any choice of $r, b \geq 0$ such that $r +b  = 4$.

\medskip\noindent
{\em Case 1:  The existence of a $( 4,0 )$-coloured copy  (the proof for a $(0,4)$-coloured copy follows by symmetry)} 

Clearly as $|R|  \geq n +1$  there is a vertex $v$  with $\deg_{red}(v)  \geq 3$.  Let $X$ be the subgraph of $K_n$ induced by the vertex-set $V(K_n)\backslash N_{red}[v]$.  We consider three possible cases:

\medskip\noindent
{\em (i) $\deg_{red}(v)  \geq 5$}

If there is a red edge in $V(K_n)  \backslash \{ v\}$  then clearly we have red $K_{1,3} \cup  K_2$,  because we always have at least three vertices in $N_{red}(v)$ disjoint from the vertices of this red edge.

 Hence the only red edges in $E(K_n)$ are those incident with $v$  and it follows that  $|R| \leq n-1$,  a contradiction.


\medskip\noindent
{\em (ii) $\deg_{red}(v)  =4$}

Clearly $|X| \geq 5$.  If there is a red edge from $N_{red}(v)$ to  $X$ we are done. So all red edges are contained in $N_{red}[v] $.  But then, $|N_{red}[v] | = 5$ hence $|R|  \leq 10  < n +1$ since $n \geq 10$.

\medskip\noindent
{\em (iii) $\deg_{red}(v) = 3$}

Clearly we may assume that all edges in $X$ are blue otherwise we are done and also  $|X| = n-4 \geq 6$.  
Since no vertex has red degree  at least four (otherwise we are back to case ii), it follows that  the four vertices in $N_{red}[v]$ are incident with at most 9 red edges altogether, a contradiction since $n \geq 10$.

\medskip\noindent
{\em Case 2 : The existence of a $( 3,1)$-coloured copy (the proof for 
a $( 1,3)$-coloured copy follows by symmetry)}

In fact we will show the existence of $(3,1)$-coloured copy with a red $K_{1,3}$  and a blue $K_2$.
As before, let $v$ be a vertex with $\deg_{red}(v)\geq 3$.

\medskip\noindent
{\em (i) $\deg_{red}(v) \geq 5$}

If there is a blue  edge in the subgraph of $K_n$ induced by $V (K_n) \backslash \{ v \}$,  then clearly we have a red $K_{1,3}$  union a blue $K_2$, because  we always have at least three vertices in Nred(v) disjoint from the vertices of this blue edge.

Hence the only blue edges in $E(Kn)$ are those incident with $v$ and it follows $|B| \leq   n-6$,  a contradiction.

\medskip\noindent
{\em (ii) $deg_{red}(v)  = 4$} 

Clearly, we can assume that all edges in $X$ are red otherwise we are done by taking a blue edge in $X$ with a red $K_{1,3}$ in  $N_{red}[v]$.  Moreover $|X|  = n – 5  \geq 5$, implying $X$ is a  complete red graph on at least five vertices.

If there is a blue edge from $N_{red}[v]$ to $X$ we are done as we can take this blue edge and a vertex-disjoint red $K_{1,3}$ in $X$.

If there is a blue edge in $N_{red}(v)$ we are done as we can take this blue edge and a red $K_{1,3}$ in $X$.   So no blue edges are possible, a contradiction to  $|B| \geq  n+1$.

\medskip\noindent
{\em (iii) $\deg_{red}(v)  = 3$} 
 Clearly we may assume that all edges in $X$ are red otherwise we are done, and also  $|X| = n-4 \geq 6$.  

If there is a blue edge in $N_{red}(v)$ we are done by taking it with the red $K_{1,3}$  from $X$. If there is a blue edge $e$ from $N_{red}[v]$ to $X$ we can still take a red $K_{1,3}$ in $X$  which is vertex-disjoint from the blue edge $e$ and we are done.  Hence no blue edges are possible, a contradiction to  $|B| \geq n +1$.

\medskip\noindent
{\em Case 3: the existence of a $(2,2)$-coloured copy}

We will show the existence of $K_{1,3}$ with two red edges and one blue edge and a vertex-disjoint blue $K_2$.
Suppose first that there is no  vertex $v$ incident with two red edges and a blue edge.  Then in each vertex either all edges are blue or all edges except one are blue or all edges are red.

Suppose there is a vertex $v$ with $\deg_{blue}(v) = n-1$.  Since there is a vertex $u$  with $\deg_{red}(v) \geq 3$ it follows that in $u$ there is  the required coloured $K_{1,3}$.

Suppose there is a vertex $v$ with $\deg_{red}(v) = n-1$.  Then there must be a red edge in the subgraph of $K_n$ induced by $V(K_n)\backslash v$,    say $e = yz$.  Then $\deg_{red}(y) = n-1$,  otherwise we are done. But that forces all vertices in $V(K_n)$ to have red-degree at least 2. Hence all edges are red, a contradiction.

The only case that remains is that all red edges are vertex disjoint but then $|R|  \leq n/2$ a contradiction.

So assume that there is a coloured copy of $K_{1,3}$ with  two red edges $va$ and $vb$ and a blue edge $vc$. Clearly all edges in $Y$, the subgraph of $K_n$ induced by $V(K_n) \backslash \{ v,a,b,c \}$  must be red otherwise we are done.  Also  $Y$ has $n - 4 \geq 6$ vertices.

If there is a blue edge $e_1$  from $a$  or from $b$ to  $z\in V(Y)$,  we are done by taking $K_{1,2}$ centered on $z$  with $e_1$ and $vc$.
So all blue edges are incident with $c$  with at most one more possible  blue edge $ab$,  making the total number of blue edges at most $n$ a contradiction.
\end{proof}

\begin{theoremnum} \label{bal4k2}
For $n \geq 10$, $\bal(n, 4K_2) =n-1$.

\end{theoremnum}

\begin{proof}
\noindent \emph{Lower bound:}

Pick a vertex $v$ in $K_n$ and colour all the edges incident to it red and the rest of the edges  blue.  Clearly (as there are no two independent edges of red colour) there is no balanced $4K_2$ hence $\bal(n,4K_2) \geq n-1$ always holds.



	\medskip\noindent \emph{Upper bound:} We have to show that  for $n \geq 10$  and any 2-edge colouring  of $K_n$ with $\min \{ R , B \}  \geq n$,  there exists a balanced $4K_2$.

	We may assume without loss of generality that we have $K_n$ with its edges coloured red and blue such that there are at least $n$ edges of each colour. Hence we may assume without loss of generality that there are two independent red edges $ab$ and $cd$. Let us remove the vertices $a, b, c, d$ and all edges incident to them to leave a graph $X$ isomorphic to $K_{n-4}$ with its edges coloured red and blue. If there are two independent blue edges in this remaining graph then we are done. So we may assume that no such two edges exist in $X$, which implies that there are at most $n-5$ blue edges in $X$, and hence at least five blue edges among the deleted edges. Therefore there exists at least one blue edge joining a vertex $u$ in $X$ to one of the vertices $a, b, c$ or $d$, say $a$, without loss of generality.We then have two cases.
	
	\medskip\noindent
	\emph{Case 1: There are no blue edges in $X$.}
	
	This implies that besides $au$ there are at least four other blue edges among the deleted edges. This gives us two subcases. Either (i) there is another blue edge joining vertex $u$ in $X$ to one of $\{b,c,d\}$; or (ii) there is a blue edge in the complete graph induced by $\{a,b,c,d\}$ and another independent edge joining $u$ to one of these vertices. We consider these two subcases separately.
	
	\emph{Subcase (i):} We may assume without loss of generality that $v$ in $X$ is joined to $b$ with a blue edge. If we remove $u, v$ and all the edges incident to them in $X$ we are left with $K_{n-6}$ in which all the edges are red. This graph has two independent red edges since ${n-6 \choose 2} \geq n-6$ because $(n-7)/2>1$ since we are assuming that $n\geq 10$. These two red edges, together with the two blue edges $au$ and $bv$ form the required balanced $4K_2$.
	
	\emph{Subcase (ii):} If we remove $u$ and its edges from $X$ we have $K_{n-5}$ whose edges are all red, and hence there are two independent red edges in this graph for $n\geq 8$, and again we can form the balanced $4K_2$ using these two edges and the two independent blue edges.
	
	\medskip\noindent
	\emph{Case 2: There is at least one blue edges, say $ux$, in $X$.}
	
	Therefore there must be a blue edge $e$ among the deleted blue edges which is not incident to $u$, since there are four such deleted edges incident to $u$, and at least another four blue edges. If $e$ is incident to $x$ then there must be another edge among the deleted ones not incident to $x$ because again there are at most three remaining deleted edges incident to $x$. So there are two independent blue edges, one of which may be $ux$. Again, if we remove $u, x$ and all edges incident to these vertices from $X$, then we have two independent red edges in this remaining graph if
	${n-6 \choose 2} - (n-5) \geq n-6$, that is, $n^2-15n+52 \geq 0$, which is true when $n\geq 10$. These two red edges, together with the two independent blue edges, form the required balanced $4K_2$.

\end{proof}

\begin{theoremnum} \label{bal2k2k12}
For $n \geq 8$, $\bal(n, 2K_2  \cup  K_{1,2} ) = 1$.
\end{theoremnum}

\begin{proof}
\noindent {\em Lower bound:}  Colour one edge of $K_n$ red and the rest blue. Clearly no copy of $2K_2 \cup  K_{1,2}$  can have two red edges.

\medskip

\noindent {\em Upper bound:}  Clearly there must be  $K_2 \cup K_{1,2}$  with the two edges of $K_{1,2}$ of distinct colour, say  without loss of generality, there are 5 vertices  $a,b,c,d,e $ such that $ab$, is red  $cd$ is red and $de$ is blue (where the vertex $d$ is the centre of $K_{1,2}$).  Let $X = V \backslash \{ a,b,c,d,e\}$,  where,   since $n \geq 8$,   $|X| \geq 3$. Let  $u_1,\ldots,u_{n-5}$ be the vertices of $X$.

If there is a blue edge in $X$ then we are done. So let us  assume all edges of $X$ are red.  

If there is a blue edge from either $a$ or $b$ to $X$, say  to $u_1$, then we are done by taking the edge $au_1$  or $ bu_1$, the edge $u_2u_3$ and $K_{1,2}$ on the vertices $\{c;d,e\}$ (where the notation $\{p:q,r,s,\ldots\}$ denotes a star with centre $p$). So let us  assume all edges from $a$ and $b$ to $x$ are red. 

If there is a blue edge from $d$ to $X$, say to $u_1$ we are done by the edges $ab$, $u_2u_3$ and $K_{1,2}$ on the vertices $\{e;d,u_1\}$.  So let us assume all edges from $d$ to $X$  are red.  

If there is a blue edge from $c$ to $X$, say to $u_1$ we are done by the edges $ab$, $ed$ and $K_{1,2}$ on the vertices $\{c;u_1,u_2\}$. So let us assume all edges from $c$ to $X$ are red.

 If there is a blue edges from $e$ to $X$  say to $u_1$ we are done by the edges $ab$, $u_2u_3$ and $K_{1,2}$ on the vertices $\{d;e,u_1\}$.  So let us assume all edges from $e$ to $X$  are red.

If edge $ce$ is  blue we are done by the edges  $ab$, $u_1u_2$ and $K_{1,2}$ on the vertices $\{  d;e,c\}$.  So assume  edge $ce$ is red.

If either edge $ae$ or $be$ is blue we are done by $ bu_1$, $cu_2$ and $K_{1,2}$ on the vertices $\{a;e,d\}$  in  the case $ae$ is  blue, or $au_1$, $cu_2$ and $K_{1,2}$ on the vertices $\{ b;e,d\}$   in the case $be$ is blue.  So let us assume these two edges are red.

 If there is a blue edge from $c$ to either $a$ or $b$, we are done by the edges $ac$, $de$ and $K_{1,2}$ on the vertices  $\{ u_1;u_2,u_3\}$ or  by $bc$, $de$ and $K_{1,2}$ on the vertices  $\{ u_1;u_2,u_3\}$.   So let us assume both these edges are red.

 Then this implies that either $da$ or $db$ is blue, as there must be at least two blue edges. If $da$ is blue we are done by $bu_1$, $cu_2$ and $K_{1,2}$ on the vertices $\{ e;d,a\} $, while if $db$ is blue we are done by $au_1$, $cu_2$ and $K_{1,2}$ on the vertices $\{ e;d,b\}$.

\end{proof}

\begin{theoremnum} \label{balk2k3}
For $n \geq 7$, $\bal(n, K_2 \cup  K_3) = 3$.
\end{theoremnum}
 
\begin{proof}

 \noindent {\em Lower bound:} Consider a coloring of $K_n$ in which there is a red coloured $K_3$ and the rest of the edges are blue. Clearly  there is no  $K_2 \cup K_3$  with exactly two red edges hence   $\bal(n , K_2 \cup K_3)  \geq 3$.

\medskip

 \noindent {\em Upper bound:} :   Suppose  now $n \geq 7$  and we have an $(r,b)$-colouring of $K_n$ with  $\min\{R ,B \}  \geq 4$.  We will show  the existence of a balanced $K_2 \cup  K_3$.  Without loss of generality, there must be a $K_3$ on vertices $\{a,b,c\}$ with two red edges $ab$ and $bc$ and one blue edge $ac$.

Let $X = V \backslash \{ a,b,c\}$. Then $|X| \geq 4$ and let $u_1,\ldots,u_{n-3}$ be the vertices of $X$.

Now if  there is a blue edge in $X$ we are done. So we may assume that all edges in $X$ are red.

Suppose there are at least  two blue edges from $b$ to $X$  say  to $u_1$ and $u_2$. Then we are done by $K_3$ induced on $\{b,u_1,u_2\}$ and the edge $u_3u_4$.

If there is one blue edge from $b$ to $X$ say to $u_1$ then we are done by $K_3$ induced on $\{b,u_1,u_2\}$ and the edge $ac$.

So we may assume $b$ is connected to $X$ by red edges only.  Suppose there are at least two blue edges from $a$ to $X$, say  to $u_1$ and $u_2$.  Then we are done by $K_3$ induced on $\{a,u_1,u_2\}$ and the edge $bc$.

Finally suppose there are at least two blue edges from $c$ to $X$, say  to $u_1$ and $u_2$.  We are done by  $K_3$ induced on $\{c,u_1,u_2\}$ and the edge $ab$.

 So there is at most one blue edge from $a$ to $X$ and one  from $c$ to $X$. But then the total number of blue edges is at most 3,  a contradiction.

The colouring with $\{ a,b,c\}$  a blue triangle and all other edges red shows that is result is best possible.
\end{proof}

\begin{theoremnum} \label{balk13k2}
For $n \geq 9$,  $\bal( n , K_{1,3} \cup  K_2)  =  n-1$.  
\end{theoremnum}

\begin{proof}

\noindent {\em Lower bound:} Clearly $\bal( n , K_{1,3} \cup  K_2) \geq n-1$:  consider $K_n$ and choose a vertex $v$.  Colour all edges incident with $v$ red and all other edges blue.   No balanced copy of  $K_{1,3} \cup K_2$ exists.

\medskip

\noindent {\em Upper bound:}  Suppose now that $\min \{  |R|  ,  |B| \}  \geq n$.  We have to show the existence of a balanced copy of  $K_{1,3} \cup K_2$ (two red edges and two blue edges).  Since $n \geq 9$ and $\min\{  |R|  ,  |B| \}  \geq n$, it follows that there are at least two independent  red, respectively blue, edges.

Clearly there is a vertex $v$ incident to  both red and blue edges. We assume, without loss of generality, that  $deg_{red}(v)  \geq deg_{blue}(v)  \geq 1$.  Clearly $deg_{red}(v) \geq \frac{ n-1}{2} \geq 4$.   We consider the following cases:  

\noindent \emph{Case 1}:   $deg_{blue}(v) = 1$ with the blue edge $e=vw$.  If there is a blue edge $e^*=xy$ where both $x$ and $y$ are in $N_{red}(v)$ then since $deg_{red}(v)  \geq 4$, we can take $e=vw$ and two  red edges incident with $v$ but vertex-disjoint from $x$ and $y$,  and add $e^*$  to get a balanced $K_{1,3} \cup K_2$.  Otherwise,  if there is no such  edge as $e^*$  then all blue edges are incident with $w$,  which would imply that  $|B| \leq n-1$, a contradiction.

 \smallskip

\noindent \emph{Case 2}: $ deg_{blue}(v)  \geq 2$ --- this forces all edges induced in $N_{red}(v)$ to be blue,  for otherwise there is a red edge say $e =xy$ in $ N_{red}(v) $ and a red edge $e^*$ incident with $v$ but not with $x$ nor $y$,  and we take two blue edges incident to  $v$  together with $e$ and $e^*$  to get a balanced $K_{1,3} \cup K_2$.  So,  all edges induced  in $N_{red}(v)$ are blue.  We take $e=xy$ to be such a blue edge and since  $deg_{red}(v) \geq 4$ we have two red edges $va$ and $vb$ which are  vertex disjoint from  $x$ and $y$.  We now take one blue edge incident with $v$ and the edges $va$  and $vb$  together with the edge $e=xy$  to get a balanced $K_{1,3} \cup K_2$. 

\end{proof}

\begin{theoremnum} \label{sbalk13}
For $n \geq 5$,  $\sbal(n,K_{1,3}) = n-1$.
\end{theoremnum} 

 \begin{proof}
\noindent {\em Lower bound:} We consider $K_n$ and fix a vertex $v$ in it, colouring all edges incident with $v$ blue and all other edges red.

 Clearly no  $K_{1,3}$  with two blue edges and one red edge exists, and hence $\sbal(n,K_{1,3}) \geq n-1$.

 \medskip

\noindent {\em Upper bound:} We now prove that if $\min \{ |R|  ,|B| \} \geq n$ (which is possible as $n \geq 5$),  there must be a $K_{1,3}$ with two blue edges and one red edge, as well as one with two red edges and one blue edge. 

So  let $n \geq 4$   and consider the blue edges. Since there are at least $n$ blue edges there must be a vertex $v$  incident with at least two blue edges.  Now if $v$ is also incident to a red edge we have a copy of $K_{1,3}$ with two blue edges and one red edge.

Hence all edges incident with $v$ are blue and hence all vertices of $K_n$ are incident to at least one blue edge.  But as there are at least $n$ blue edges there must be another vertex $u$ already incident to $v$ such that $uv$ is blue, but also incident with another vertex $w$ with $uw$ coloured blue.  Now if $u$ is also incident to a red edge we are done with $K_{1,3}$ centered at $u$ having two blue edges and one red edge.  Otherwise all edges incident to $u$  are blue and it follows that all vertices of $K_n$ have degree at least 2 in the blue graph. But there must be a red edge incident with some vertex $z$, and we have $K_{1,3}$ centred at $z$ with two blue edges and one red edge.  The case for $K_{1,3}$ with two red and one blue edge follows by symmetry atrting with a $K_{1,3}$ with at least two red edges.

\end{proof}

\noindent Finally we shall find this simple observation about graphs on three edges useful for results given in the tables.

\begin{observation} \label{3edges}

If $G$ has three edges then $G$ is balanceable.

\end{observation}

\begin{proof}
Suppose $\min\{|R|,|B|\} \geq 1$ then clearly there is a vertex $v$ in $K_n$  incident with a red edge say $vx$  and blue edge say $vy$.  So no matter how we complete this red and blue coloured $K_{1,2}$  to any of the graphs on 3 edges except for $3K_2$, we have a balanced copy of $G$.  If $G  = 3K_2$  and $n \geq 6$ then as before there is red and blue coloured $K_{1,2}$  with vertices $\{v; x ,y\}$.  Since $n \geq 6$   there is an edge  $ab$ disjoint from $\{v ,x ,y\}$  hence without loss of generality we may assume that edge $vx$  is red and edge $ab$  is blue.  Since $n \geq 6$, every edge disjoint from $\{a ,b ,v ,x\}$ gives a balanced $3K_2$.
\end{proof}

\section{The Triple Property and Union of Bipartite Graphs }

It is well known \cite{2018arXivCHM}  that there are bipartite graphs which are not balanceable  however the following theorem shows an interesting property:  if $G$ and $H$ are any  bipartite graphs   with $e(G) = e(H)$  then $G \cup H$  is balanceable and even more.  The example following this theorem shows a direct use of this theorem in the case where $G  =  2C_{4t +2}$, while it is  known   that   $C_{4t +2}$ is non-balanceable \cite{2018arXivCHM}.

We next develop the triple property which together with  Theorem \ref{thm:ramsey} allows us to compute $\sbal(n, (2t-1)K_2)$,   $\bal( n ,2tK_2)$ and $\bal( n, (2t +1)K_2)$  in one stroke.

 \begin{theoremnum} \label{thm:ramsey}
Suppose $G$  and $H$ are bipartite graphs such that $e(G) = e(H)$ and that  $n \geq |G|+|H|+ R(G,H)$  where $R(G,H)$ is the Ramsey number for a red copy of $G$ or a blue copy of $H$.  Then, for $n \geq n_0$,  $\bal(n,G \cup H)  \leq  \max\{  \ex(n,G) , \ex(n,H)  \}$.
\end{theoremnum}

\begin{proof}

Since $G$ and $H$ are bipartite graphs and it is well-known that  $\ex(n,G)$ and $\ex(n,H)$  are sub-quadratic \cite{furedi2013history},  it follows that, for $n$ large enough, if $\min \{ |R|  ,  |B| \}  >  \max\{ \ex(n,G) , \ex(n,H)\}$  then a 2-edge-coloured copy of $K_n$ contains both a red copy of $G$ and a blue copy of $H$, \cite{bowen2019colored,2018arXivCHM,girao2019tur}.

Let $G_1$ be the red copy of $G$ and $H_1$  be the  blue copy of $H$.  If  $G_1$ and $H_1$ are vertex-disjoint we are done, having a balanced $G \cup H$  since $ e(G) = e(H)$.

Otherwise let  $S = V(G_1) \cup  V(H_1)$.  Clearly $|S| < |V(G)|+ |V(H)|$.  Let  $X = V(K_n) \backslash S$. Then $|X| \geq  R(G,H)$  and hence in the induced colouring on $X$ there is either a red copy of $G$ or a blue copy of $H$ (or both).

If there is a blue copy of $H$ take it with $G_1$, and if there is a red copy of $G$ take it with $H_1$ and in both cases we get a balanced $G \cup H$.  
\end{proof}
 
We now give an example of the applicability of this Theorem.

\begin{Ex}
An illustration of  Theorem \ref{thm:ramsey}.
\end{Ex}

It is known that  $C_{4n +2}$  is not balanceable   by  remark 2.9  in \cite{2018arXivCHM}.  Also  $R(C_{4t +2}, C_{4t +2}  ) =   6t +2$ by a result in \cite{faudree1974all}.  Also Turan numbers for even cycles are bounded above by  $\ex(n,C_{2k}) \leq  (k - 1)n^{1+1/k}+ 16(k - 1)n$, a result proved in \cite{pikhurko2012note}.

So applying Theorem \ref{thm:ramsey}, together with these facts, we get the following:

Let $G =  2C_{4t +2}$  then $G$  is balanceable  and for $n \geq 14t+6$,  \[\bal(n,G)  \leq  \ex(n,C_{4t +2})  \leq    (4t+1)n^{1+1/(4t+2)}+ 16(4t+ 1)n.\]
 
\bigskip
 
We define the following property:  Let $G$, $H$ and $F$ be three graphs such that $e(G)$ is odd, $H$ is obtained from $G$ by adding a new edge, and $F$ is obtained from $H$ by adding another new edge.  We say that $(G,H,F)$ has \emph{the triple property} if there is some $n_0$ such that for $n \geq n_0$, $\sbal(n,G)=\bal(n,H)=\bal(F,n)$.

\begin{observation} \label{triplep}
Let $G=(2t-1)K_2$, $H=2tK_2$ and $F=(2t+1)K_2$. Then $(G,H,F)$ has the triple property.
\end{observation}
\begin{proof}

 Clearly $\bal(n,F) \geq \bal(n,H)$, for suppose we have a colouring of $E(K_n)$ with $\min\{|R|,|B|\} \geq \bal(n,F)+1$.  Then by definition there is a balanced $(r,b)$-coloured copy of $F$ with either $r=t+1$ and $b=t$ or $r=t$ and $b=t+1$.  Ignoring an edge with the most frequent colour gives a balanced colouring of $H$.

Conversely, suppose we have a colouring of $E(K_n)$ with $\min\{|R|,|B|\} \geq \bal(n,H)+1$.  Then by definition there is a balance $(r,b)$-coloured copy of $H$ with $r=b=t$.  If we take $n>n_0= 4t+2$, the edges of the balanced copy of $H$ cover $4t$ vertices, but there is at least one further edge independent from all these $2t$ edges, and no matter what the colour of this edge is, we can add it to $H$ to get a balanced $F$. Therefore $\bal(n,F) \leq \bal(n,H)$.

We now need to show that $\bal(n,H) \geq \sbal(n,G)$.  Suppose we have a colouring of $E(K_n)$ with $\min\{|R|,|B|\} \geq \bal(n,H)+1$.  Then by definition there is a balanced copy of $H$, and we can drop either a red or a blue edge to get a balanced copy of $G$.

For the converse  consider $n_0   =  5t + 2   =  n(G) + t +4$.  Suppose we have a colouring of $E(K_n)$  with  $\min\{ |R|  ,|B| \}  \geq  \sbal(n,G) +1$. Then by definition there is either a $  (t, t-1)$-coloured copy of $G$ or a $( t-1,t)$-coloured copy of $G$.

 Consider a $( t , t-1)$-coloured copy of $G$.  Let the t red edges be $e_1, \ldots,e_t$   and the $t-1$ blue edges be $ f_1,\ldots,f_{ t-1}$.  Let    $S$ be the complete graph induced by  $V(K_n)  \backslash  V(G)$; clearly  $|V(S)| \geq t + 4$  since $n_0 \geq  5t +2$.   Clearly all edges in $S$ must be red for otherwise we can add a blue edge to get a balanced copy of $H$.

If there is a blue edge not incident with any of $f_1, \ldots,f_{t-1}$,  then  either it is adjacent to two red edges of $e_1,\ldots,e_t$, or there is a blue edge adjacent with an edge from $e_1,\ldots,e_t$ and an edge from $S$.  In the first case  we drop these two red edges and add the blue edge and two independent red edges from  $S$  since $|S)| \geq t + 4$.  In the second case  we drop the red edge and replace it by a red edge from $S$ not incident with the blue edge and add also the blue edge to get balanced $H$.

So we know that $S$ contains only red edges, therefore together with $e_1,\ldots, e_t$, we have a graph $L$ on at least $3t+4$ vertices, all of whose edges are red.

So we may conclude that there are remaining blue edges  and  that each is adjacent to at least one blue edge from $f_1, \ldots, f_{t-1}$.

Now take an $( t-1 , t )$-coloured  copy of $G$.   The $t$ blue edges  are incident with at most $t$ vertices from $L$,  leaving in $L$ at least $2t +4$ vertices from which we can choose $t$ red edges  not adjacent with the $t$ blue edges of $G$ and we get a balanced $H$.
\end{proof}

\begin{theoremnum} \label{cor:ramsey}

For $ n \geq  7t-1 $,   $\sbal(n, (2t - 1)K_2) =  \bal(n, 2tK_2)  = \bal( n,(2t +1)K_2 ) =  \ex(n,tK_2)  =  \binom{ t-1}{2} + (t-1)(n-t+1)$.

 \end{theoremnum}

\begin{proof}

It suffices to prove  that $\bal(n,2tK_2 ) = \ex(n,tK_2 )$   for $n  =  7t -1$   (by the former triple property all other equality signs were proved).

From  the theorem above we know  that $\bal(n,2tK_2) \leq \ex(n,tK_2)$.   However in this case  the reverse inequality holds as well  because we take  $K_n$ and colour its edges with $|R|  = \ex(n,tK_2)$ forming the extremal graph for $tK_2$  (not having $tK_2$)  .

If there is a balanced $2tK_2$ it must contains a red $tK_2$ which  is a contradiction.

We observe an old result of Erdos and  Gallai \cite{erdHos1959maximal}   \[\ex(n,tK_2 )   = \max \{  \binom{2t  - 1}{2}, \binom{ t  - 1}{2} + ( t-1 )(n-t +1)  \} =  \binom{t-1}{2} +  (t-1)(n-t +1)\]  for $n \geq \frac{7t-6}{2}$.

Also, we observe that  $R(tK_2 , tK_2) =  3t  - 1$  by the classical result  in \cite{cockayne1975ramsey}.

Hence putting all these facts together we got the last required equality for $n \geq  2|V(G)| + R(G,G)  =  4t  + 3t  - 1  = 7t  - 1$.

\end{proof}

\newpage

\section{Tables}

We can now give, in this section, the values of $\ot(n,G)$, $\bal(n,G)$ and  $\sbal(n,G)$, when they exist, for all graphs $G$ on at most four edges.

\bigskip

\begin{threeparttable}
\begin{tabular}{lllll p{1 in}}
	Graphs & Amoeba & Omnitonal & ot$(n,G)$ & Valid $n$ & Comments \\ \hline
	$4K_2$ & Y & Y & ex$(n,G)$ & $n \geq n_0$ & Theorem \ref{bipam}\\
	$2K_2\cup K_{1,2}$ & Y & Y & ex$(n,G)$ & $n \geq n_0$ &  Theorem \ref{bipam}\\
	$2K_{1,2}$ & Y & Y & ex$(n,G)$ & $n \geq n_0$ &  Theorem \ref{bipam}\\
	$K_2\cup P_4$ & Y & Y & ex$(n,G)$ & $n \geq n_0$ &  Theorem \ref{bipam}\\
	$P_5$ & Y & Y & ex$(n,G)$ &$n \geq n_0$ &  Theorem \ref{bipam}\\
	$K_{1,3}$ with extended leaf & Y & Y & ex$(n,G)$ & $n \geq n_0$ &  Theorem \ref{bipam}\\
	$K_2\cup K_3$ & N & N & & --- &  Theorem \ref{omnibip} \\
	$C_{4}$ & N  \tnote{1}& N & &--- &Lemma \ref{typecolouring}  \\	
	$K_{1,3}\cup K_2$ & N \tnote{1} & Y & n & $n\geq 10$ & Theorem \ref{omnithm}\\
	$K_{1,4}$ & N \tnote{1}& Y & $2n-3$ & $n\geq 16$ & Theorem \ref{omnistar} \\
	$K_3 + e$ & Y & N & & --- & Theorem \ref{omnibip}\\ \hline
	$K_{1,3}$ & N  \tnote{1} & Y & $n$ & $n\geq 12$ & Theorem \ref{omnistar}\\
	$P_4$ & Y & Y & ex$(n,G)$ & $n \geq n_0$ & Theorem \ref{bipam}\\
	$K_3$ &N \tnote{1} &N & &--- &  Theorem \ref{omnibip} \\
	$3K_2$ &Y &Y & ex$(n,G)$ & $n \geq n_0$ & Theorem \ref{bipam}\\
	$P_3\cup K_2$ &Y &Y & ex$(n,G)$ & $n \geq n_0$ & Theorem \ref{bipam}\\ \hline

	$P_3$ & Y & Y & ex$(n,G)$ & $n \geq 3$ & Theorem \ref{bipam}\\
	$2K_2$ &Y &Y & ex$(n,G)$ & $n \geq n_0$ & Theorem \ref{bipam}\\
	\hline
\end{tabular}
\begin{tablenotes}
\item[1] \footnotesize{By Lemma \ref{lemmadeg}}
\end{tablenotes}
\caption{Amoebas and Omnitonal graphs on at most four edges}
\end{threeparttable}

\bigskip

\begin{threeparttable}
\begin{tabular}{lll p{1.5 in}}
	Graphs & bal$(n,G)$ & Valid $n$ & Comments\\ \hline
	$4K_2$  & $n-1$ & $n \geq 10$ & Theorem \ref{bal4k2} \\
	$2K_2\cup K_{1,2}$   & 1  & $n \geq 8$ & Theorem  \ref{bal2k2k12}\\
	$2K_{1,2}$ &  1& $n \geq 7$ & $\mbox{ }$\tnote{1} \\
	$K_2\cup P_4$ &1& $n \geq 7$ & $\mbox{ }$\tnote{1} \\
	$P_5$ & 1 & $n \geq 6$ & $\mbox{ }$\tnote{1}  \\
	$K_{1,3}$ with extended leaf  &  1 &  $n \geq 7$ & $\mbox{ }$\tnote{1} \\ 
	$K_2\cup K_3$ & 3 & $n \geq 7$ & Theorem \ref{balk2k3} \\
	$C_{4}$ &1&  $n \geq 4$ & $\mbox{ }$\tnote{1}\\	
	$K_{1,3}\cup K_2$  & $n-1$ & $n \geq 9$  & Theorem \ref{balk13k2} \\
	$K_{1,4}$ & $n-1$  & $n \geq 5 $ & Theorem E\\
	$K_3 + e$  &1 & $n \geq 5$ & $\mbox{ }$\tnote{1}\\ \hline
	$K_{1,3}$  &0& &Observation \ref{3edges}\\
	$P_4$  &0 & &Observation \ref{3edges}\\
	$K_3$  &0 & &Observation \ref{3edges}\\
	$3K_2$  &0 & & Observation \ref{3edges}\\
	$P_3\cup K_2$  & 0& &Observation \ref{3edges}\\
	\hline
	\hline

\end{tabular}	
\begin{tablenotes}
\item[1] \footnotesize{The proofs are in nature very similar to the proof of Theorem \ref{bal2k2k12} and are left to the interested reader to verify.}
\end{tablenotes}
\caption{Balanced graphs on at most four edges}
\end{threeparttable}

\bigskip

\begin{threeparttable}
\begin{tabular}{llll p{1 in}}
	Graphs  & Strong balanced & sbal$(n,G)$ & Valid $n$ & Comments \\ \hline
	$K_{1,3}$    & Y   & $n-1$& $n \geq 4$ & Theorem \ref{sbalk13}\\
	$P_4$ &  Y  & 1  & $n \geq 7$ & Theorem \ref{ambal}  \\
	$K_3$  &  N & & &  Lemma \ref{typecolouring} \\
	$3K_2$  & Y  & $n-1$  & $n \geq 7$ & Theorem \ref{cor:ramsey}\\
	$P_3\cup K_2$  &Y  & 1  &  $n \geq 7$ & Theorem \ref{ambal}\\

	\hline

\end{tabular}	
\caption{Strongly balanced graphs on at most four edges}
\end{threeparttable}
\section{Conclusion}

In this paper, by computing the values of $\bal(n,G)$, $\sbal(n,G)$ and $\ot(n,G)$ for all graphs $G$ on at most four edges we have tried to convey the flavour of the results in \cite{ 2018arXivCHM} and the techniques used to obtain them.  We have also tried to obtain some new techniques which could shed more insight on these problems.  We hope that this paper will be an invitation to the interested reader to delve into \cite{2018arXivCHM} for a more comprehensive treatment of balanceable and omnitonal graphs.

\section*{Acknowledgement}

We would like to thank an anonymous referee whose careful reading of the paper helped us improve it considerably.

\nocite{caro2016ero}
\nocite{caro2019zero1}
\nocite{caro2019zero2}

\bibliographystyle{plain}
\bibliography{mbmbib1}

\begin{thebibliography}{10}

\bibitem{augspurger2016avoiding}
C.~Augspurger, M.~Minter, K.~Shoukry, P.~Sissokho, and K.~Voss.
\newblock Avoiding zero-sum subsequences of prescribed length over the
  integers.
\newblock {\em arXiv preprint arXiv:1603.03978}, 2016.

\bibitem{berger2019analogue}
A.~Berger.
\newblock An analogue of the {E}rd{\H{o}}s--{G}inzburg--{Z}iv theorem over
  $\mathbb{Z}$.
\newblock {\em Discrete Mathematics}, 342(3):815--820, 2019.

\bibitem{bollobas2004extremal}
B.~Bollob{\'a}s.
\newblock {\em Extremal graph theory}.
\newblock Courier Corporation, 2004.

\bibitem{bowen2019colored}
M.~Bowen, A.~Hansberg, A.~Montejano, and A.~M{\"u}yesser.
\newblock Colored unavoidable patterns and balanceable graphs.
\newblock {\em arXiv preprint arXiv:1912.06302}, 2019.

\bibitem{caro2019amoeba}
Y.~Caro, A.~Hansberg, and A.~Montejano.
\newblock Amoebas.
\newblock {\em in preparation}, 2019.

\bibitem{2018arXivCHM}
Y.~Caro, A.~Hansberg, and A.~Montejano.
\newblock {Unavoidable chromatic patterns in 2-colorings of the complete
  graph}.
\newblock {\em arXiv e-prints}, page arXiv:1810.12375, 2019.

\bibitem{caro2019zero2}
Y.~Caro, A.~Hansberg, and A.~Montejano.
\newblock Zero-sum ${K}_m$ over $\mathbb{Z}$ and the story of ${K}_4$.
\newblock {\em Graphs and Combinatorics}, 35(4):855--865, 2019.

\bibitem{caro2019zero1}
Y.~Caro, A.~Hansberg, and A.~Montejano.
\newblock Zero-sum subsequences in bounded-sum $\{- 1, 1\}$-sequences.
\newblock {\em Journal of Combinatorial Theory, Series A}, 161:387--419, 2019.

\bibitem{caro2016ero}
Y.~Caro and R.~Yuster.
\newblock On zero-sum and almost zero-sum subgraphs over {$\mathbb{Z}$}.
\newblock {\em Graphs and Combinatorics}, 32(1):49--63, 2016.

\bibitem{cockayne1975ramsey}
E.J. Cockayne and P.J. Lorimer.
\newblock The ramsey number for stripes.
\newblock {\em Journal of the Australian Mathematical Society}, 19(2):252--256,
  1975.

\bibitem{erdHos1959maximal}
P.~Erd{\H{o}}s and T.~Gallai.
\newblock On maximal paths and circuits of graphs.
\newblock {\em Acta Mathematica Hungarica}, 10(3-4):337--356, 1959.

\bibitem{faudree1974all}
R.J. Faudree and R.H. Schelp.
\newblock All ramsey numbers for cycles in graphs.
\newblock {\em Discrete Mathematics}, 8(4):313--329, 1974.

\bibitem{furedi2013history}
Z.~F{\"u}redi and M.~Simonovits.
\newblock The history of degenerate (bipartite) extremal graph problems.
\newblock In {\em Erd{\H{o}}s Centennial}, pages 169--264. Springer, 2013.

\bibitem{girao2019tur}
A.~Gir{\~a}o and B.~Narayanan.
\newblock Tur{\'a}n theorems for unavoidable patterns.
\newblock {\em arXiv preprint arXiv:1907.00964}, 2019.

\bibitem{pikhurko2012note}
O.~Pikhurko.
\newblock A note on the {T}ur{\'a}n function of even cycles.
\newblock {\em Proceedings of the American Mathematical Society},
  140(11):3687--3692, 2012.

\bibitem{robertson2018zerosum2}
A.~Robertson.
\newblock Zero-sum analogues of van der waerden's theorem on arithmetic
  progressions.
\newblock {\em arXiv preprint arXiv:1802.03387}, 2018.

\bibitem{robertson2018zerosum1}
A.~Robertson.
\newblock Zero-sum generalized schur numbers.
\newblock {\em arXiv preprint arXiv:1802.03382}, 2018.

\bibitem{sun2019zero}
A.~Sun.
\newblock Zero-sum subsequences in bounded-sum $\{-r, s \}$-sequences.
\newblock {\em arXiv preprint arXiv:1907.06623}, 2019.

\bibitem{west2017introduction}
D.B. West.
\newblock {\em Introduction to Graph Theory}.
\newblock Math Classics. Pearson, 2017.

\end{thebibliography}

\end{document}